\documentclass[11pt]{amsart}

\theoremstyle{plain}
\newtheorem{theorem}                {Theorem}      [section]

\newtheorem*{theorem2}                {Theorem \ref{thm:gap2}}
\newtheorem*{theorem3}                {Theorem \ref{thm:surfaces}}
\newtheorem{proposition}  [theorem]  {Proposition}
\newtheorem{corollary}    [theorem]  {Corollary}
\newtheorem{lemma}        [theorem]  {Lemma}

\theoremstyle{definition}

\newtheorem{remark}         {Remark}[section]
\newtheorem{definition}   [theorem]  {Definition}

\DeclareMathOperator{\trace}{trace} 
\DeclareMathOperator{\grad}{grad}

 \DeclareMathOperator{\id}{I}
 
\DeclareMathOperator{\ric}{Ric}

\DeclareMathOperator{\Span}{span}
 
\DeclareMathOperator{\cst}{constant}
 
 \DeclareMathOperator{\im}{Im}

\numberwithin{equation}{section}

\begin{document}

\title[Biharmonic submanifolds with parallel mean curvature]
{Biharmonic submanifolds with parallel mean curvature in $\mathbb{S}^n\times\mathbb{R}$}

\author{Dorel~Fetcu}
\author{Cezar~Oniciuc}
\author{Harold~Rosenberg}

\address{Department of Mathematics\\
"Gh. Asachi" Technical University of Iasi\\
Bd. Carol I no. 11 \\
700506 Iasi, Romania} \email{dfetcu@math.tuiasi.ro; dorel@impa.br}

\address{Faculty of Mathematics\\ "Al.I. Cuza" University of Iasi\\
Bd. Carol I no. 11 \\
700506 Iasi, Romania} \email{oniciucc@uaic.ro}

\address{IMPA\\ Estrada Dona Castorina\\ 110, 22460-320 Rio de
Janeiro, Brasil} \email{rosen@impa.ro}

\thanks{The first author was supported by a Post-Doctoral Fellowship "P\'os-Doutorado S\^enior
(PDS)" offered by FAPERJ, Brazil.}

\begin{abstract} We find a Simons type formula for submanifolds with parallel mean curvature vector (pmc submanifolds)
in product spaces $M^n(c)\times\mathbb{R}$, where $M^n(c)$ is a
space form with constant sectional curvature $c$, and then we use it to prove
a gap theorem for the mean curvature of certain complete proper-biharmonic pmc
submanifolds, and classify proper-biharmonic pmc surfaces in
$\mathbb{S}^n(c)\times\mathbb{R}$.
\end{abstract}



\maketitle

\section{Introduction}

The notion of \textit{biharmonic maps} was suggested in 1964 by Eells and Sampson in
\cite{JEJS}, as a natural generalization of \textit{harmonic maps}.
Thus, whilst a harmonic map $\psi:(M,g)\rightarrow(\bar M,h)$ between two Riemannian manifolds
is defined as a critical point of the \textit{energy functional}
$$
E(\psi)=\frac{1}{2}\int_{M}|d\psi|^{2}\ v_{g},
$$
a biharmonic map is a critical point of the
\textit{bienergy functional}
$$
E_{2}(\psi)=\frac{1}{2}\int_{M}|\tau(\psi)|^{2}\ v_{g},
$$
where $\tau(\psi)=\trace\nabla d\psi$ is the tension field that
vanishes for harmonic maps. The Euler-Lagrange equation for the
bienergy functional was derived by Jiang in 1986 (see \cite{GYJ}):
\begin{align*}
\tau_{2}(\psi)&=\Delta\tau(\psi)-\trace\bar R(d\psi,\tau(\psi))d\psi\\
&=0
\end{align*}
where $\tau_{2}(\psi)$ is the \textit{bitension field} of $\psi$, $\Delta=\trace(\nabla^{\psi})^2 =\trace(\nabla^{\psi}\nabla^{\psi}-\nabla^{\psi}_{\nabla})$ is the rough Laplacian defined on
sections of $\psi^{-1}(T\bar M)$ and
$\bar R$ is the curvature tensor of $\bar M$, given by $\bar R(X,Y)Z=[\bar\nabla_X,\bar\nabla_Y]Z-\bar\nabla_{[X,Y]}Z$. Since any harmonic map
is biharmonic, we are interested in non-harmonic biharmonic maps,
which are called \textit{proper-biharmonic}.

A \textit{biharmonic submanifold} in a Riemannian manifold is a
submanifold for which the inclusion map is biharmonic. In Euclidean
space the biharmonic submanifolds are the same as those defined by
Chen in \cite{BYC}, as they are characterized by the equation
$\Delta H=0$, where $H$ is the mean curvature vector field and
$\Delta$ is the rough Laplacian.

Some very fertile environments for finding examples of proper-biharmonic submanifolds proved to be the unit Euclidian sphere $\mathbb{S}^n$, and, in general, space forms with positive sectional curvature. For example, whilst there are no proper-bi\-har\-mo\-nic curves and surfaces in $3$-dimensional spaces with non-positive constant sectional curvature (see Chen and Ishikawa's paper \cite{CI} and Dimitric's paper \cite{D} in the case of Euclidian space, and Caddeo, Montaldo and Oniciuc's article \cite{CMO} when the sectional curvature is negative) we do have examples of such submanifolds in $\mathbb{S}^3$ in \cite{CMO1}, where they are explicitly classified.

In the very recent paper \cite{OW}, Ou and Wang studied the
biharmonicity of constant mean curvature surfaces (cmc surfaces) in
Thurston's $3$-dimensional geometries, amongst them being the
product space $\mathbb{S}^2\times\mathbb{R}$.

The case of cmc surfaces in product spaces of type $M^2(c)\times\mathbb{R}$, where $M^2(c)$ is a simply connected surface with constant sectional curvature $c$, and then that of surfaces with parallel mean curvature vector field (pmc surfaces) in product spaces of type $M^n(c)\times\mathbb{R}$, where $M^n(c)$ is a space form, received a special attention (see, for example, Abresch and Rosenberg's papers \cite{AR,AR2} on cmc surfaces, and Alencar, do Carmo and Tribuzy's article \cite{AdCT} on pmc surfaces). From the point of view of biharmonicity, pmc surfaces and, in general, pmc submanifolds in spheres, were studied in \cite{CMO} in \cite{BO}, respectively.

In his paper \cite{JS} from 1968, Simons proved a very important formula for the
Laplacian of the second fundamental form of a minimal submanifold in
a Riemannian manifold and then used it to characterize certain
minimal submanifolds of a sphere and Euclidean space. Over the years, such formulas, called Simons type equations, also proved to be a powerful tool for studying cmc and pmc submanifolds.

In our paper, we first obtain a Simons type equation for pmc
submanifolds in product spaces $M^n(c)\times\mathbb{R}$ and then we
use it to prove a gap phenomenon for the mean curvature of a
proper-biharmonic pmc submanifold. We also investigate the
biharmonicity of pmc surfaces in product spaces and, using a
reduction of codimension result of Eschenburg and Tribuzy in
\cite{ET} and the above mentioned Simons type formula, we get a
classification theorem. Our main results are the following two
theorems.

\begin{theorem2} Let $\Sigma^m$ be a complete proper-biharmonic pmc submanifold
in $\mathbb{S}^n\times\mathbb{R}$, with $m\geq 2$, such that its
mean curvature satisfies
$$
|H|^2>\frac{(m-1)(m^2+4)+(m-2)\sqrt{(m-1)(m-2)(m^2+m+2)}}{2m^3},
$$
and the norm of its second fundamental form $\sigma$ is bounded.
Then $m<n$, $|H|=1$ and $\Sigma^m$ is a minimal submanifold of a
small hypersphere $\mathbb{S}^{n-1}(2)\subset\mathbb{S}^n$.
\end{theorem2}

\begin{theorem3} Let $\Sigma^2$ be a proper-biharmonic pmc surface in $\mathbb{S}^n(c)\times\mathbb{R}$. Then either
\begin{enumerate}
\item $\Sigma^2$ is a minimal surface of a small hypersphere $\mathbb{S}^{n-1}(2c)\subset\mathbb{S}^n(c)$$;$ or

\item $\Sigma^2$ is an $($an open part of$)$ a vertical cylinder $\pi^{-1}(\gamma)$, where $\gamma$ is
a circle in $\mathbb{S}^2(c)$ with curvature equal to $\sqrt{c}$, i.e. $\gamma$ is a biharmonic circle in $\mathbb{S}^2(c)$.
\end{enumerate}
\end{theorem3}

\noindent \textbf{Acknowledgments.} The first author would like to
thank the IMPA in Rio de Janeiro for providing a very stimulative work environment during the preparation of this paper.

\section{Preliminaries}

Let $M^n(c)$ be a space form, i.e. a simply-connected $n$-dimensional manifold with
constant sectional curvature $c$, and consider the product manifold
$\bar M=M^n(c)\times\mathbb{R}$. The expression of the curvature
tensor $\bar R$ of such a manifold can be obtained from
$$
\langle\bar R(X,Y)Z,W\rangle=c\{\langle d\pi Y, d\pi Z\rangle\langle d\pi X, d\pi W\rangle-\langle d\pi X, d\pi Z\rangle\langle d\pi Y, d\pi W\rangle\},
$$
where $\pi:\bar M=M^n(c)\times\mathbb{R}\rightarrow M^n(c)$ is the projection map. After a straightforward computation we get
\begin{equation}\label{eq:barR}
\begin{array}{ll}
\bar R(X,Y)Z=&c\{\langle Y, Z\rangle X-\langle X, Z\rangle Y-\langle Y,\xi\rangle\langle Z,\xi\rangle X+\langle X,\xi\rangle\langle Z,\xi\rangle Y\\ \\&+\langle X,Z\rangle\langle Y,\xi\rangle\xi-\langle Y,Z\rangle\langle X,\xi\rangle\xi\},
\end{array}
\end{equation}
where $\xi$ is the unit vector tangent to $\mathbb{R}$.

Let $\Sigma^m$, $m\leq n$, be an $m$-dimensional submanifold of $\bar M$. From the equation of Gauss
$$
\begin{array}{ll}
\langle R(X,Y)Z,W\rangle=&\langle\bar R(X,Y)Z,W\rangle\\ \\&+\sum_{\alpha=m+1}^{n+1}\{\langle A_{\alpha}Y,Z\rangle\langle A_{\alpha}X,W\rangle-\langle A_{\alpha}X,Z\rangle\langle A_{\alpha}Y,W\rangle\},
\end{array}
$$
we obtain the expression of its curvature tensor
\begin{equation}\label{eq:R}
\begin{array}{ll}
R(X,Y)Z=&c\{\langle Y, Z\rangle X-\langle X, Z\rangle Y-\langle Y,T\rangle\langle Z,T\rangle X+\langle X,T\rangle\langle Z,T\rangle Y\\ \\&+\langle X,Z\rangle\langle Y,T\rangle T-\langle Y,Z\rangle\langle X,T\rangle T\}\\ \\&+\sum_{\alpha=m+1}^{n+1}\{\langle A_{\alpha}Y,Z\rangle A_{\alpha}X-\langle A_{\alpha}X,Z\rangle A_{\alpha}Y\},
\end{array}
\end{equation}
where $T$ is the component of $\xi$ tangent to $\Sigma^m$, $A$ is the shape operator defined by the equation of Weingarten
$$
\bar\nabla_XV=-A_VX+\nabla^{\perp}_XV,
$$
for any vector field $X$ tangent to $\Sigma^m$ and any normal vector
field $V$. Here $\bar\nabla$ is the Levi-Civita connection on $\bar
M$, $\nabla^{\perp}$ is the connection in the normal bundle, and
$A_{\alpha}=A_{E_{\alpha}}$, $\{E_{\alpha}\}_{\alpha=m+1}^{n+1}$
being a local orthonormal frame field in the normal bundle.

\begin{definition} A submanifold $\Sigma^m$ of $M^n(c)\times\mathbb{R}$ is called a \textit{vertical cylinder} over $\Sigma^{m-1}$ if $\Sigma^m=\pi^{-1}(\Sigma^{m-1})$, where $\pi:M^n(c)\times\mathbb{R}\rightarrow M^n(c)$ is the projection map and $\Sigma^{m-1}$ is a submanifold of $M^n(c)$.
\end{definition}

It is easy to see that vertical cylinders $\Sigma^m=\pi^{-1}(\Sigma^{m-1})$ are characterized by the fact that $\xi$ is tangent to $\Sigma^m$.

\begin{definition} If the mean curvature vector field $H$ of a submanifold $\Sigma^m$ is
parallel in the normal bundle, i.e.
$\nabla^{\perp}H=0$, then $\Sigma^m$ is called a \textit{pmc submanifold}.
\end{definition}

\begin{remark}\label{rem:cyl} It is straightforward to verify that $\Sigma^m=\pi^{-1}(\Sigma^{m-1})$
is a pmc vertical cylinder in $M^n(c)\times\mathbb{R}$ if and only
if $\Sigma^{m-1}$ is a pmc submanifold in $M^n(c)$. Moreover, the
mean curvature vector field of $\Sigma^m$ is $H=\frac{n-1}{n} H_0$,
where $H_0$ is the mean curvature vector field of $\Sigma^{m-1}$. It
also easy to prove that the vertical cylinder
$\Sigma^m=\pi^{-1}(\Sigma^{m-1})$ is proper-biharmonic in
$M^n(c)\times\mathbb{R}$ if and only if $\Sigma^{m-1}$ is
proper-biharmonic in $M^n(c)$.
\end{remark}

We end this section by recalling the following two results, which we
shall use later.

\begin{lemma}[\cite{AdC,O}]\label{l:oku} Let $a_i$, $i=1,\ldots,m$, be real numbers such that $\sum_{i=1}^ma_i=0$ and $\sum_{i=1}^ma_i^2=b^2$, where $b=\cst\geq 0$. Then
$$
-\frac{m-2}{\sqrt{m(m-1)}}b^3\leq\sum_{i=1}^ma_i^3\leq\frac{m-2}{\sqrt{m(m-1)}}b^3,
$$
and equality holds in the right-hand $($left-hand$)$ side if and only if $(n-1)$ of the $a_i$'s are non-positive and equal $($$(n-1)$ of the $a_i$'s are non-negative and equal$)$.
\end{lemma}

\begin{theorem}[Omori-Yau Maximum Principle, \cite{Y}]\label{OY} If $\Sigma^m$ is a complete Riemannian
manifold with Ricci curvature bounded from below, then for any
smooth function $u\in C^2(\Sigma^m)$ with $\sup_{\Sigma^m}
u<+\infty$ there exists a sequence of points
$\{p_k\}_{k\in\mathbb{N}}\subset \Sigma^m$ satisfying
$$
\lim_{k\rightarrow\infty}u(p_k)=\sup_{\Sigma^m}u,\quad |\nabla u|(p_k)<\frac{1}{k}\quad\textnormal{and}\quad\Delta u(p_k)<\frac{1}{k}.
$$

\end{theorem}

\section{A Simons type formula for submanifolds in $M^n(c)\times\mathbb{R}$}\label{s1}

Let $\Sigma^m$, $m\leq
n$, be an $m$-dimensional submanifold of
$M^n(c)\times\mathbb{R}$, with mean curvature vector field $H$. In this section we shall compute the Laplacian of the squared
norm of $A_V$, where $V$ is a normal vector field to the submanifold, such
that $V$ is parallel in the normal bundle, i.e. $\nabla^{\perp}V=0$,
and $\trace A_V=\cst$.

\begin{lemma}\label{lemma:com}
If $U$ and $V$ are normal vector fields to $\Sigma^m$ and $V$ is parallel
in the normal bundle, then $[A_V,A_U]=0$, i.e. $A_V$ commutes with
$A_U$.
\end{lemma}

\begin{proof}
The conclusion follows easily, from the Ricci equation,
$$
\langle R^{\perp}(X,Y)V,U\rangle=\langle[A_V,A_U]X,Y\rangle+\langle\bar R(X,Y)V,U\rangle,
$$
since $R^{\perp}(X,Y)V=0$ and \eqref{eq:barR} implies that $\langle\bar R(X,Y)V,U\rangle=0$.
\end{proof}

Now, from the Codazzi equation,
$$
\begin{array}{cl}
\langle \bar R(X,Y)Z,V\rangle=&\langle\nabla^{\perp}_X\sigma(Y,Z),V\rangle-\langle\sigma(\nabla_XY,Z),V\rangle
-\langle\sigma(Y,\nabla_XZ),V\rangle\\ \\&-\langle\nabla^{\perp}_Y\sigma(X,Z),V\rangle+\langle\sigma(\nabla_YX,Z),V\rangle
+\langle\sigma(X,\nabla_YZ),V\rangle,
\end{array}
$$
where $\sigma$ is the second fundamental form of $\Sigma^m$, we get
$$
\begin{array}{rl}
\langle \bar R(X,Y)Z,V\rangle=&X(\langle A_VY,Z\rangle)-\langle\sigma(Y,Z),\nabla^{\perp}_XV\rangle
-\langle A_V(\nabla_XY),Z\rangle\\ \\&-\langle A_VY,\nabla_XZ\rangle-Y(\langle A_VX,Z\rangle)+\langle\sigma(X,Z),\nabla^{\perp}_YV\rangle
\\ \\&+\langle A_V(\nabla_YX),Z\rangle+\langle A_VX,\nabla_YZ\rangle\\ \\=&\langle(\nabla_XA_V)Y-(\nabla_YA_V)X,Z\rangle,
\end{array}
$$
since $\nabla^{\perp}V=0$. Therefore, using \eqref{eq:barR}, we obtain
\begin{equation}\label{eq:Codazzi}
(\nabla_XA_V)Y=(\nabla_YA_V)X+c\langle V,N\rangle(\langle Y,T\rangle X-\langle X,T\rangle Y),
\end{equation}
where $N$ is the normal part of $\xi$.

Next, we have the following Weitzenb\"ock fromula
\begin{equation}\label{eq:Laplacian}
\frac{1}{2}\Delta|A_V|^2=|\nabla A_V|^2+\langle\trace\nabla^2A_V,A_V\rangle,
\end{equation}
where we extended the metric $\langle,\rangle$ to the tensor space in the standard way.

The second term in the right hand side of \eqref{eq:Laplacian} can
be calculated by using a method introduced in \cite{NS}, and, in the following, for
the sake of completeness, we shall sketch this computation.

Let us consider
$$
C(X,Y)=(\nabla^2 A_V)(X,Y)=\nabla_X(\nabla_YA_V)-\nabla_{\nabla_XY}A_V,
$$
and note that we have the following Ricci commutation
formula
\begin{equation}\label{eq:C} C(X,Y)=C(Y,X)+[R(X,Y),A_V].
\end{equation}
Next, consider an orthonormal basis $\{e_i\}_{i=1}^{m}$ in
$T_p\Sigma^m$, $p\in\Sigma^m$, extend $e_i$ to vector fields $E_i$
in a neighborhood of $p$ such that
$\{E_i\}$ is a geodesic frame field around $p$, and let us denote $X=E_k$. We have
$$
(\trace\nabla^2A_V)X=\sum_{i=1}^mC(E_i,E_i)X.
$$

Using equation \eqref{eq:Codazzi}, we get, at $p$,
$$
\begin{array}{ll}
C(E_i,X)E_i&=\nabla_{E_i}((\nabla_{X}A_V)E_i)\\ \\&=\nabla_{E_i}((\nabla_{E_i}A_V)X)+c\nabla_{E_i}(\langle V,N\rangle(\langle E_i,T\rangle X-\langle X,T\rangle E_i))
\end{array}
$$
and then
\begin{equation}\label{eq:1}
\begin{array}{lcl}
C(E_i,X)E_i&=&C(E_i,E_i)X-c\langle A_VE_i,T\rangle(\langle E_i,T\rangle X-\langle X,T\rangle E_i)\\ \\&&+c\langle V,N\rangle(\langle A_NE_i,E_i\rangle X-\langle A_NX,E_i\rangle E_i),
\end{array}
\end{equation}
where we used $\sigma(E_i,T)=-\nabla^{\perp}_{E_i}N$ and
$\nabla_{E_i}T=A_NE_i$, which follow from the fact that $\xi$ is
parallel, i.e. $\bar\nabla\xi=0$.

We also have, at $p$,
\begin{equation}\label{eq:2}
C(X,E_i)E_i=\nabla_{X}((\nabla_{E_i}A_V)E_i),
\end{equation}
and, from \eqref{eq:C}, \eqref{eq:1} and \eqref{eq:2}, we get, also at $p$,
$$
\begin{array}{lcl}
C(E_i,E_i)X&=&\nabla_{X}((\nabla_{E_i}A_V)E_i)+[R(E_i,X),A_V]E_i\\ \\&&+c\langle A_VE_i,T\rangle(\langle E_i,T\rangle X-\langle X,T\rangle E_i)\\ \\&&-c\langle V,N\rangle(\langle A_NE_i,E_i\rangle X-\langle A_NX,E_i\rangle E_i).
\end{array}
$$
Since $\nabla_{E_i}A_V$ is symmetric, from \eqref{eq:Codazzi}, one obtains
$$
\begin{array}{lcl}
\langle\sum_{i=1}^m(\nabla_{E_i}A_V)E_i,Z\rangle&=&\sum_{i=1}^m\langle E_i,(\nabla_{E_i}A_V)Z\rangle=\sum_{i=1}^m\langle E_i,(\nabla_{Z}A_V)E_i\rangle\\ \\&&+c\langle V,N\rangle\sum_{i=1}^m\langle E_i,\langle Z,T\rangle E_i-\langle E_i,T\rangle Z\rangle\\ \\&=&\trace(\nabla_ZA_V)+c(m-1)\langle V,N\rangle\langle T,Z\rangle\\ \\&=&Z(\trace A_V)+c(m-1)\langle V,N\rangle\langle T,Z\rangle\\ \\&=&c(m-1)\langle V,N\rangle\langle T,Z\rangle,
\end{array}
$$
for any vector $Z$ tangent to $\Sigma^m$, since $\trace A_V=\cst$.

From the Gauss equation \eqref{eq:R} of the surface $\Sigma^2$, and Lemma \ref{lemma:com}, we get, after a straightforward computation,
$$
\begin{array}{lcl}
\sum_{i=1}^mR(E_i,X)A_VE_i&=&c\{A_VX-(\trace A_V)X+(\trace A_V)\langle X,T\rangle
T\\ \\&&-\langle A_VX,T\rangle T
-\langle X,T\rangle A_VT+\langle A_VT,T\rangle X\}
\\ \\&&+\sum_{\alpha=m+1}^{n+1}\{A_VA_{\alpha}^2 X-(\trace(A_VA_{\alpha}))A_{\alpha}X\},
\end{array}
$$
and
$$
\begin{array}{lcl}
\sum_{i=1}^mA_VR(E_i,X)E_i&=&-c\{(m-1-|T|^2)A_VX-(m-2)\langle X,T\rangle A_VT\}\\ \\&&+\sum_{\alpha=m+1}^{n+1}\{A_VA_{\alpha}^2X-(\trace A_{\alpha})A_VA_{\alpha}X\}.
\end{array}
$$

Therefore, we have
$$
\begin{array}{lcl}
(\trace\nabla^2A_V)X&=&\sum_{i=1}^m C(E_i,E_i)X\\ \\&=&\sum_{i=1}^m[R(E_i,X),A_V]E_i\\ \\
&&+c\{m\langle V,N\rangle A_NX-(m-1)\langle A_VX,T\rangle T+\langle A_VT,T\rangle X\\ \\&&-\langle X,T\rangle A_VT-m\langle H,N\rangle\langle V,N\rangle X\}\\ \\
&=&c\{(m-|T|^2)A_VX+2\langle A_VT,T\rangle X-m\langle A_VX,T\rangle T\\ \\&&-m\langle X,T\rangle A_VT+m\langle V,N\rangle A_NX-m\langle H,N\rangle\langle V,N\rangle X\\ \\&&-(\trace A_V)X+(\trace A_V)\langle X,T\rangle T\}\\ \\&&+\sum_{\alpha=m+1}^{n+1}\{(\trace A_{\alpha})A_VA_{\alpha}X-(\trace(A_VA_{\alpha}))A_{\alpha}X\},
\end{array}
$$
and then
$$
\begin{array}{lcl}
\langle\trace\nabla^2A_V,A_V\rangle&=&\sum_{i=1}^m\langle(\trace\nabla^2A_V)E_i,A_VE_i\rangle\\ \\&=&c\{(m-|T|^2)|A_V|^2-2m|A_VT|^2+3(\trace A_V)\langle A_VT,T\rangle\\ \\&&+m(\trace(A_NA_V))\langle V,N\rangle-(\trace A_V)^2\\ \\&&-m(\trace A_V)\langle H,N\rangle\langle V,N\rangle\}\\ \\&&+\sum_{\alpha=m+1}^{n+1}\{(\trace A_{\alpha})(\trace(A_V^2A_{\alpha}))-(\trace(A_VA_{\alpha}))^2\}.
\end{array}
$$

Thus, from \eqref{eq:Laplacian}, we obtain the following proposition.

\begin{proposition}\label{p:delta} Let $\Sigma^m$ be a submanifold of $M^n(c)\times\mathbb{R}$, with mean curvature vector field $H$ and shape operator $A$. If $V$ is a normal vector field, parallel in the normal bundle, with $\trace A_V=\cst$, then
\begin{equation}\label{eq:delta2}
\begin{array}{lcl}
\frac{1}{2}\Delta|A_V|^2&=&|\nabla A_V|^2+c\{(m-|T|^2)|A_V|^2-2m|A_VT|^2\\ \\&&+3(\trace A_V)\langle A_VT,T\rangle\\ \\&&+m(\trace(A_NA_V))\langle V,N\rangle-(\trace A_V)^2\\ \\&&-m(\trace A_V)\langle H,N\rangle\langle V,N\rangle\}\\ \\&&+\sum_{\alpha=m+1}^{n+1}\{(\trace A_{\alpha})(\trace(A_V^2A_{\alpha}))-(\trace(A_VA_{\alpha}))^2\},
\end{array}
\end{equation}
where $\{E_{\alpha}\}_{\alpha=m+1}^{n+1}$ is a local orthonormal frame field in the normal bundle.
\end{proposition}

\section{A gap theorem for biharmonic pmc submanifolds in $\mathbb{S}^n\times\mathbb{R}$}

Whereas complete biharmonic pmc submanifolds of
$\mathbb{S}^n\times\mathbb{R}$ are the subject of our first main
theorem, we have the following result for compact submanifolds.

\begin{proposition}\label{thm:gap} If $\Sigma^m$ is a compact biharmonic submanifold in $\mathbb{S}^n(c)\times\mathbb{R}$, then $\Sigma^{m}$ lies in $\mathbb{S}^n(c)$.
\end{proposition}

\begin{proof} The \textit{height function} of a submanifold $\Sigma^m$ in $\mathbb{S}^n(c)\times\mathbb{R}$ is defined by
$$
h=t\circ i:\Sigma^m\rightarrow\mathbb{R},
$$
where $t:\mathbb{S}^n(c)\times\mathbb{R}\rightarrow\mathbb{R}$ is
the projection map and $i:\Sigma^m\rightarrow
\mathbb{S}^n(c)\times\mathbb{R}$ is the inclusion map. It is easy to
verify that
$$
\tau(h)=dt(\tau(i))\quad\textnormal{and}\quad\tau_2(h)=dt(\tau_2(i)),
$$
and we see that, if $\Sigma^m$ is biharmonic, then $h$ is also a
biharmonic function.
 
Since $\Sigma^m$ is a compact biharmonic submanifold, it follows
that $h$ is a real valued biharmonic function defined on a compact
manifold, which, according to a result in \cite{GYJ}, leads to the
fact that $h$ is actually a harmonic function, but then, using the
maximum principle, we get that $h$ is constant, i.e.  $\Sigma^{m}$
lies in $\mathbb{S}^n(c)$.
\end{proof}

We recall now the following three results which we shall use later
in this paper.

\begin{theorem}[\cite{On}]\label{thm:teza} A proper-biharmonic cmc submanifold $\Sigma^m$ in
$\mathbb{S}^n(c)$, with mean curvature equal to $\sqrt{c}$, is
minimal in a small hypersphere
$\mathbb{S}^{n-1}(2c)\subset\mathbb{S}^n(c)$.
\end{theorem}

\begin{theorem}[\cite{BO}]\label{thm:BO} If $\Sigma^m$ is a
proper-biharmonic pmc submanifold in $\mathbb{S}^n(c)$, with mean
curvature vector field $H$ and $m>2$, then
$|H|\in\big(0,\frac{m-2}{m}\sqrt{c}\big]\cup\{\sqrt{c}\}$. Moreover,
$|H|=\frac{m-2}{m}\sqrt{c}$ if and only if $\Sigma^m$ is $($an open
part of$)$ a standard product
$$
\Sigma_1^{m-1}\times\mathbb{S}^1(2c)\subset\mathbb{S}^n(c),
$$
where $\Sigma_1^{m-1}$ is a minimal submanifold in
$\mathbb{S}^{n-2}(2c)$.
\end{theorem}

\begin{theorem}[\cite{BMO}] A submanifold $\Sigma^m$ in a Riemannian manifold $\bar M$, with second fundamental form $\sigma$,
mean curvature vector field $H$, and shape operator $A$, is biharmonic if and only if
\begin{equation}\label{eq:bih}
\begin{cases}
-\Delta^{\perp}H+\trace\sigma(\cdot,A_H\cdot)+\trace(\bar R(\cdot,H)\cdot)^{\perp}=0\\
\frac{m}{2}\grad|H|^2+2\trace A_{\nabla^{\perp}_{\cdot}H}(\cdot)+2\trace(\bar R(\cdot,H)\cdot)^{\top}=0,
\end{cases}
\end{equation}
where $\Delta^{\perp}$ is the Laplacian in the normal bundle and $\bar R$ is the curvature tensor of $\bar M$.
\end{theorem}

Now, we have the following two corollaries.

\begin{corollary}\label{coro:bih_pmc}  A pmc submanifold $\Sigma^m$ in $M^n(c)\times\mathbb{R}$, with $m\geq 2$, is biharmonic if and only if
\begin{equation}\label{eq:bih_pmc}
\begin{cases}
H\perp\xi\\|A_H|^2=c(m-|T|^2)|H|^2\\ \trace(A_HA_U)=0\quad\textnormal{for any normal vector}\quad U\perp H.
\end{cases}
\end{equation}
\end{corollary}

\begin{proof} Since $\Sigma^m$ is a pmc submanifold, equations \eqref{eq:bih}
become
$$
\begin{cases}
\trace\sigma(\cdot,A_H\cdot)+\trace(\bar R(\cdot,H)\cdot)^{\perp}=0\\
\trace(\bar R(\cdot,H)\cdot)^{\top}=0
\end{cases}
$$
and, as from equation \eqref{eq:barR} we have
$$
\trace \bar R(\cdot,H)\cdot=c\{(m-1)\langle H,\xi\rangle T-(m-|T|^2)H+m\langle H,\xi\rangle N\},
$$
we see that $\Sigma^m$ is biharmonic if and only if
$$
\trace\sigma(\cdot,A_H\cdot)=c\{(m-|T|^2)H-m\langle H,\xi\rangle N\}\quad\textnormal{and}\quad\langle H,\xi\rangle T=0.
$$

Now, assume that there exists a point $p\in\Sigma^m$ such that $\langle H,\xi\rangle\neq 0$ at $p$, and then
$\langle H,\xi\rangle\neq 0$ on a neighborhood of $p$. It follows that $T=0$ on this neighborhood, i.e. $\langle X,\xi\rangle=0$ for any tangent vector field $X$. Since $\bar\nabla\xi=0$, we have
$$
0=\langle\bar\nabla_YX,\xi\rangle=\langle\sigma(X,Y),\xi\rangle
$$
for any tangent vector fields $X$ and $Y$. Thus $\langle H,\xi\rangle=0$ on a neighborhood of $p$, and, therefore, at $p$, which is a contradiction.
Consequently, we have that $H\perp\xi$ everywhere on $\Sigma^m$. Then, one obtains
$$
\trace\sigma(\cdot,A_H\cdot)=c(m-|T|^2)H,
$$
from where we get \eqref{eq:bih_pmc}.
\end{proof}

\begin{remark} A direct consequence of Corollary \ref{coro:bih_pmc} is that there are no proper-biharmonic pmc submanifolds in a product space $M^n(c)\times\mathbb{R}$ with $c\leq 0$.
\end{remark}

\begin{corollary}\label{rem:hyper} If $\Sigma^n$ is a proper-biharmonic cmc hypersurface
in $\mathbb{S}^n(c)\times\mathbb{R}$, then it
is $($an open part of$)$ a vertical cylinder $\pi^{-1}(\Sigma^{n-1})$, where
$\Sigma^{n-1}$ is a proper-biharmonic cmc hypersurface in
$\mathbb{S}^n(c)$. Moreover, if
\begin{enumerate}

\item $n=2$, then $\Sigma^1$ is a circle in $\mathbb{S}^2(c)$ with curvature equal to $\sqrt{c}$, and $|H|=\frac{1}{2}\sqrt{c}$$;$

\item $n=3$, then $\Sigma^2$ is an open part of a small hypersphere $\mathbb{S}^2(2c)\subset\mathbb{S}^3(c)$, and $|H|=\frac{2}{3}\sqrt{c}$$;$

\item $n>3$, then $|H|\in\big(0,\frac{n-3}{n}\sqrt{c}\big]\cup\big\{\frac{n-1}{n}\sqrt{c}\big\}$.
Furthermore,
\begin{enumerate}

\item[(a)] $|H|=\frac{n-3}{n}\sqrt{c}$ if and only if $\Sigma^{n-1}$ is an open part of the standard product $\mathbb{S}^{n-2}(2c)\times\mathbb{S}^1(2c)\subset\mathbb{S}^n(c)$$;$

\item[(b)] $|H|=\frac{n-1}{n}\sqrt{c}$ if and only if $\Sigma^{n-1}$ is an open part of a small hypersphere $\mathbb{S}^{n-1}(2c)\subset\mathbb{S}^n(c)$.
\end{enumerate}
\end{enumerate}
\end{corollary}

\begin{proof} From Corollary \ref{coro:bih_pmc}, we get that the mean
curvature vector field $H$ of our submanifold is orthogonal to
$\xi$, which means that $\xi$ is tangent to $\Sigma^n$. Therefore,
$\Sigma^n$ is a vertical cylinder $\Sigma^{n-1}\times\mathbb{R}$,
where $\Sigma^{n-1}$ is a proper-biharmonic cmc hypersurface in
$\mathbb{S}^n(c)$, with mean curvature vector field $H_0$ satisfying
$H=\frac{n-1}{n} H_0$, as we know from Remark \ref{rem:cyl}.

Now, when $n\in\{2,3\}$, the main result in \cite{CMP} and
\cite[Theorem~4.8]{CMO1} lead to $(1)$ and $(2)$, respectively,
and when $n>3$, we use Theorem \ref{thm:BO} to prove $(3)$.
\end{proof}

\begin{proposition} Let $\Sigma^m$ be a proper-biharmonic pmc submanifold in $\mathbb{S}^n(c)\times\mathbb{R}$, with $m\geq 2$. Then its second fundamental form $\sigma$ satisfies $|\sigma|^2\geq c(m-1)$, and the equality holds if and only if $\Sigma^m$ is a vertical cylinder $\pi^{-1}(\Sigma^{m-1})$ in $\mathbb{S}^{m}(c)\times\mathbb{R}$, where $\Sigma^{m-1}$ is a proper biharmonic cmc hypersurface in $\mathbb{S}^{m}(c)$.
\end{proposition}

\begin{proof} From the first equation of \eqref{eq:bih_pmc}, we have
$$
|\sigma|^2\geq|A_{\frac{H}{|H|}}|^2=c(m-|T|^2)\geq c(m-1).
$$

Thus, $|\sigma|^2=c(m-1)$ if and only if $|T|=1$ at every point on $\Sigma^m$, i.e. $\Sigma^m$ is a vertical cylinder $\pi^{-1}(\Sigma^{m-1})$, and $A_V=0$ for any normal vector field $V$ orthogonal to $H$.

Next, let us consider the subbundle $L=\Span\{\im\sigma\}$, and we
see that $L$ is parallel in the normal bundle and $\dim L=1$, since
actually $L=\Span\{H\}$, and that $\bar R(X,Y)Z\in T\Sigma^m\oplus
L$, for any $X,Y,Z\in T\Sigma^m\oplus L$. Therefore, using
\cite[Theorem~2]{ET}, we get that
the cylinder $\Sigma^m$ lies in an
$(m+1)$-dimensional totally geodesic submanifold of
$\mathbb{S}^n(c)\times\mathbb{R}$, i.e. it is a vertical cylinder in
$\mathbb{S}^{m}(c)\times\mathbb{R}$.
\end{proof}

\begin{proposition} Let $\Sigma^m$ be a proper-biharmonic pmc submanifold in $\mathbb{S}^n(c)\times\mathbb{R}$, with $m\geq 2$. Then its mean curvature satisfies $|H|^2\leq c$, and the equality holds if and only if $\Sigma^m$ is minimal in a small hypersphere $\mathbb{S}^{n-1}(2c)\subset\mathbb{S}^n(c)$.
\end{proposition}

\begin{proof} Since $|A_H|^2\geq m|H|^4$, from the first equation of \eqref{eq:bih_pmc}, we get that
$$
c(m-|T|^2)\geq m|H|^2,
$$
and then $|H|^2\leq c$. The equality holds if and only if $T=0$,
which means that $\Sigma^m$ lies in $\mathbb{S}^n$. Thus, using
Theorem \ref{thm:teza}, we come to the conclusion.
\end{proof}

Now, for the sake of simplicity, we shall consider only the case
$c=1$, and we are ready to prove the first of our main results.

\begin{theorem}\label{thm:gap2} Let $\Sigma^m$ be a complete proper-biharmonic pmc submanifold
in $\mathbb{S}^n\times\mathbb{R}$, with $m\geq 2$, such that its
mean curvature satisfies
\begin{equation}\label{eq:ineq}
|H|^2>C(m)=\frac{(m-1)(m^2+4)+(m-2)\sqrt{(m-1)(m-2)(m^2+m+2)}}{2m^3},
\end{equation}
and the norm of its second fundamental form $\sigma$ is bounded.
Then $m<n$, $|H|=1$ and $\Sigma^m$ is a minimal submanifold of a
small hypersphere $\mathbb{S}^{n-1}(2)\subset\mathbb{S}^n$.
\end{theorem}

\begin{proof} From Corollary \ref{coro:bih_pmc}, we have that $\langle
H,\xi\rangle=0$, which implies
$$
0=\langle \bar\nabla_XH,\xi\rangle=-\langle A_HX,T\rangle=-\langle A_HT,X\rangle
$$
for any tangent vector field $X$, and then $A_HT=0$. Therefore, if we consider a local orthonormal frame field $\big\{E_{m+1}=\frac{H}{|H|},\ldots, E_{n+1}\big\}$ in the normal bundle, using Proposition \ref{p:delta} and equation \eqref{eq:bih_pmc}, we get
\begin{equation}\label{eq:deltaAH}
\frac{1}{2}\Delta |A_H|^2=|\nabla A_H|^2+m(\trace A_H^3)-m^2|H|^4.
\end{equation}

Let us consider $\phi_H=A_H-|H|^2\id$ the traceless part of
$A_H$. We have
$$
\trace A_H^3=\trace \phi_H^3+3|H|^2|\phi_H|^2+m|H|^6,
$$
and, using the first equation of \eqref{eq:bih_pmc},
$$
|\phi_H|^2=|A_H|^2-m|H|^4=(m-|T|^2)|H|^2-m|H|^4.
$$
Replacing in equation \eqref{eq:deltaAH}, one obtains
$$
\frac{1}{2}\Delta |\phi_H|^2=|\nabla \phi_H|^2+m(\trace\phi_H^3)+3m|H|^2|\phi_H|^2-m^2|H|^4(1-|H|^2).
$$
Using Lemma \ref{l:oku}, we get
$$
\trace\phi_H^3\geq-\frac{m-2}{\sqrt{m(m-1)}}|\phi_H|^3,
$$
and then, since $|T|^2|H|^4=|\phi_HT|^2\leq|T|^2|\phi_H|^2$,
\begin{equation}\label{eq:ineqdelta}
\begin{array}{ll}
\frac{1}{2}\Delta |\phi_H|^2&\geq-\frac{m(m-2)}{\sqrt{m(m-1)}}|\phi_H|^3+3m|H|^2|\phi_H|^2-m^2|H|^4(1-|H|^2)\\ \\
&=-\frac{m(m-2)}{\sqrt{m(m-1)}}|\phi_H|^3+2m|H|^2|\phi_H|^2-m|T|^2|H|^4\\ \\
&\geq-\frac{m(m-2)}{\sqrt{m(m-1)}}|\phi_H|^3+2m|H|^2|\phi_H|^2-m|T|^2|\phi_H|^2\\ \\
&=m|\phi_H|^2\Big(-\frac{m-2}{\sqrt{m(m-1)}}|\phi_H|+2|H|^2-|T|^2\Big).
\end{array}
\end{equation}

Now, we shall split our study in two cases, as $m\geq 3$ or $m=2$.

\noindent\textbf{Case I: $m\geq 3$.} If $2|H|^2-|T|^2>0$, then we can write
$$
-\frac{m-2}{\sqrt{m(m-1)}}|\phi_H|+2|H|^2-|T|^2=\frac{1}{m(m-1)}\frac{P(|T|^2)}{\frac{m-2}{\sqrt{m(m-1)}}|\phi_H|+2|H|^2-|T|^2},
$$
where $P(t)$ is a polynomial with constant coefficients, given by
$$
P(t)=m(m-1)t^2-(3m^2-4)|H|^2t+m|H|^2(m^2|H|^2-(m-2)^2).
$$

By using elementary arguments, we obtain that, if $|H|^2>C(m)$,
then $P(t)\geq P(1)>0$ for any $t\in(-\infty,1]$.

Since $C(m)>\frac{1}{2}$ for any $m\geq 3$, our hypothesis $|H|^2>C(m)$ implies that $2|H|^2-|T|^2>0$, and then, from \eqref{eq:ineqdelta}, we get
\begin{equation}\label{delta:poly}
\begin{array}{ll}
\frac{1}{2}\Delta|\phi_H|^2&\geq \frac{mP(|T|^2)}{\sqrt{m(m-1)}((m-2)|\phi_H|+\sqrt{m(m-1)}(2|H|^2-|T|^2))}|\phi_H|^2\\ \\
&\geq \frac{P(|T|^2)}{\sqrt{m-1}|H|((m-2)\sqrt{1-|H|^2}+2\sqrt{m-1}|H|)}|\phi_H|^2\\ \\&\geq \frac{P(1)}{\sqrt{m-1}|H|((m-2)\sqrt{1-|H|^2}+2\sqrt{m-1}|H|)}|\phi_H|^2\\ \\&\geq 0.
\end{array}
\end{equation}

Next, let us consider a local orthonormal frame field $\{E_i\}_{i=1}^{m}$ on $\Sigma^m$, $X$ a unit tangent vector field, and $\big\{E_{m+1}=\frac{H}{|H|},\ldots,E_{n+1}\big\}$ an orthonormal frame field in the normal bundle. Using equation \eqref{eq:R}, we can compute the Ricci curvature of our submanifold
$$
\begin{array}{lll}
\ric X&=&\sum_{i=1}^m\langle R(E_i,X)X,E_i\rangle\\ \\
&=&\sum_{i=1}^m\{|X|^2-\langle X,E_i\rangle^2-\langle X,T\rangle^2+2\langle X,T\rangle\langle T, E_i\rangle\langle X, E_i\rangle\\ \\&&-\langle T, E_i\rangle|X|^2
+\sum_{\alpha=m+1}^{n+1}(\langle A_{\alpha}E_i,E_i\rangle\langle A_{\alpha}X,X\rangle-\langle A_{\alpha}X,E_i\rangle^2)\}
\\ \\&=&
m-1-|T|^2-(m-2)\langle X,T\rangle^2+m\langle A_HX,X\rangle-\sum_{\alpha=m+1}^{n+1}|A_{\alpha}X|^2,
\end{array}
$$
and then, it follows that
$$
\begin{array}{ll}
\ric X&\geq (m-1)(1-|T|^2)-m|A_HX|-\sum_{\alpha=m+1}^{n+1}|A_{\alpha}|^2\\ \\ &\geq-m|A_H|-|\sigma|^2.
\end{array}
$$
Since by hypothesis we know that
$|\sigma|$ is bounded, we can see that the Ricci curvature of $\Sigma^m$ is bounded
from below, and then the Omori-Yau Maximum Principle holds on our
submanifold.

Therefore, we can use Theorem \ref{OY} with $u=|\phi_H|^2$.
It follows that there exists a sequence of points
$\{p_k\}_{k\in\mathbb{N}}\subset \Sigma^m$ satisfying
$$
\lim_{k\rightarrow\infty}|\phi_H|^2(p_k)=\sup_{\Sigma^m}|\phi_H|^2\quad\textnormal{and}\quad\Delta|\phi_H|^2(p_k)<\frac{1}{k}.
$$
Since $P(1)>0$, from \eqref{delta:poly}, we get that
$0=\lim_{k\rightarrow\infty}|\phi_H|^2(p_k)=\sup_{\Sigma^m}|\phi_H|^2$,
which means that $\phi_H=0$, i.e. $\Sigma^m$ is pseudo-umbilical.

Now, since $A_HT=0$, we have $0=A_HT=|H|^2T$, i.e. $T=0$ on
$\Sigma^m$, and therefore $\Sigma^m$ lies in $\mathbb{S}^n$, which
also implies that $m<n$. Since
$|H|^2>C(m)>(\frac{m-1}{m})^2>(\frac{m-2}{m})^2$, using Theorems
\ref{thm:teza} and \ref{thm:BO}, we come to the conclusion.

\noindent\textbf{Case II: $m=2$.} In this case, from equation \eqref{eq:ineqdelta}, we have
$$
\frac{1}{2}\Delta|\phi_H|^2\geq 2|\phi_H|^2(2|H|^2-|T|^2)=\frac{2|\phi_H|^2}{|H|^2}(|\phi_H|^2+2|H|^2(2|H|^2-1)).
$$
Now, since $|H|^2>C(2)=\frac{1}{2}$, working as in the first case, we conclude.
\end{proof}

\begin{remark} We note that, in the case of proper-biharmonic pmc
surfaces in $\mathbb{S}^n\times\mathbb{R}$, if we take $|H|^2\geq
C(2)$, then the conclusion of Theorem \ref{thm:gap2} remains
unchanged.
\end{remark}

\section{Biharmonic pmc surfaces in $\mathbb{S}^n(c)\times\mathbb{R}$}

Before proving the second main theorem of this paper we need some preliminary results.

First, we note that the map $p\in\Sigma^2\rightarrow(A_H-\mu\id)(p)$, where $\mu$ is a constant,
is analytic, and, therefore, either $\Sigma^2$ is a pseudo-umbilical surface (at every point), or $H(p)$ is not an umbilical direction for any point $p$, or $H(p)$ is an umbilical direction on a closed set without interior points. We shall denote by $W$ the set of points where $H$ is not an umbilical direction. In the second case, $W$ coincides with $\Sigma^2$, and in the third one, $W$ is an open dense set in $\Sigma^2$.

As the authors observed in \cite[Lemma 1]{AdCT}, we have that, if $\Sigma^2$ is a pmc surface in $\mathbb{S}^n(c)\times\mathbb{R}$, with mean curvature vector field $H$, then either $\Sigma^2$ is
pseudo-umbilical, i.e. $H$ is an umbilical direction everywhere, or, at any point in $W$, there exists a local orthonormal frame field that diagonalizes $A_U$ for
any normal vector field $U$ defined on $W$.

If $\Sigma^2$ is a pseudo-umbilical pmc surface in $\mathbb{S}^n(c)\times\mathbb{R}$, then it was proved in \cite[Lemma 3]{AdCT} that it lies in $\mathbb{S}^n(c)$, and, therefore, $\Sigma^2$ is minimal in a small hypersphere of $\mathbb{S}^n(c)$.

\begin{lemma}\label{lemma:5.1} Let $\Sigma^2$ be a pmc surface in $\mathbb{S}^n(c)\times\mathbb{R}$. Then $\Sigma^2$ is proper-biharmonic if and only if either
\begin{enumerate}
\item $\Sigma^2$ is pseudo-umbilical and, therefore, it is a minimal surface of a small hypersphere $\mathbb{S}^{n-1}(2c)\subset\mathbb{S}^n(c)$$;$ or

\item the mean curvature vector field $H$ is orthogonal to $\xi$, $|A_H|^2=c(2-|T|^2)|H|^2$, and $A_U=0$ for any normal vector field $U$ orthogonal to $H$.
\end{enumerate}
\end{lemma}

\begin{proof} As we have seen, in the first case, $\Sigma^2$ is a minimal surface in a small hypersphere of $\mathbb{S}^n(c)$, and then the conclusion follows from \cite[Theorem~3.4]{CMO}.

Assume now that $\Sigma^2$ is not pseudo-umbilical. In the following, we shall work on the set $W$ defined above. Let $p$ be an arbitrary point in $W$ and consider $\{e_1,e_2\}$ an orthonormal basis at $p$ that diagonalizes $A_H$ and $A_U$ for any normal vector $U$ orthogonal to $H$. Since $H\perp U$, it follows that $\trace A_U=2\langle H,U\rangle=0$. The matrices of $A_H$ and $A_U$ with respect to $\{e_1, e_2\}$ are
$$
A_H=\left(\begin{array}{cc}a+|H|^2&0\\ \\0&-a+|H|^2\end{array}\right)\quad\textnormal{and}\quad
A_U=\left(\begin{array}{cc}b&0\\ \\0&-b\end{array}\right),
$$
and then, from the last biharmonic condition \eqref{eq:bih_pmc}, we get
$0=\trace(A_HA_U)=2ab$. Since $a\neq 0$, we get $b=0$, i.e. $A_U=0$.

Finally, we extend the result by continuity throughout $\Sigma^2$, and we conclude.
\end{proof}

\begin{corollary}\label{lemma:T=ct} If $\Sigma^2$ is a proper-biharmonic pmc surface in $\mathbb{S}^n(c)\times\mathbb{R}$ then the tangent part $T$ of $\xi$ has constant length.
\end{corollary}

\begin{proof} If the surface is pseudo-umbilical, then $T=0$.

Now, assume that $\Sigma^2$ is non-pseudo-umbilical and we shall work on $W$. Let $p$ be an arbitrary point in $W$ and $X\in T_p\Sigma^2$. Since $\bar\nabla\xi=0$ and $H\perp N$, we get that
$\nabla_X T=A_NX=0$. Then, we have
$$
X(|T|^2)=2\langle\nabla_XT, T\rangle=0.
$$
By continuity, it follows that $X(|T|^2)=0$ for any tangent vector field $X$ defined on $\Sigma^2$, and we come to the conclusion.
\end{proof}

\begin{remark} We note that, if $\Sigma^2$ is a proper-biharmonic pmc surface in $\mathbb{S}^n(c)\times\mathbb{R}$ with $T=0$, then it lies in $\mathbb{S}^n(c)$ and is pseudo-umbilical (see \cite{BMO1}).
\end{remark}

We recall now the following two results.

\begin{lemma}[\cite{AdCT}]\label{lemma:non-umb} Let $\Sigma^2$ be a non-pseudo-umbilical pmc surface
in $M^n(c)\times\mathbb{R}$, with second fundamental form $\sigma$,
and define on $W$ the subbundle $L=\Span\{\im\sigma\cup N\}$ of the
normal bundle. Then $L$ is parallel, i.e. if $U$ is a smooth section on $L$, then
$\nabla^{\perp}U\in L$.
\end{lemma}

\begin{proposition}[\cite{FR}]\label{prop:deltaT} If $\Sigma^2$ is a pmc
surface in $M^n(c)\times\mathbb{R}$, then
$$
\frac{1}{2}\Delta|T|^2=|A_N|^2+K|T|^2+2T(\langle H,N\rangle),
$$
where $K$ is the Gaussian curvature of the surface.
\end{proposition}

\begin{corollary} If $\Sigma^2$ is a non-pseudo-umbilical proper-biharmonic pmc surface in $\mathbb{S}^n(c)\times\mathbb{R}$, then it is flat.
\end{corollary}

In the following, let $\Sigma^2$ be a non-pseudo-umbilical proper-bi\-har\-mo\-nic pmc surface in $\mathbb{S}^n(c)\times\mathbb{R}$. It follows that $|T|=\cst\neq 0$, i.e. $|N|=\cst\in[0,1)$. Working on the set $W$, since
$A_U=0$ for any normal vector
field $U$ orthogonal to $H$, we obtain
$\dim\Span\{\im\sigma\}=1$ and then, $\dim L=2$.
Now, we apply \cite[Theorem~2]{ET} and obtain that $W$, and therefore $\Sigma^2$, lies
in $\mathbb{S}^3(c)\times\mathbb{R}$.

Further, we shall prove that $|T|=1$ on $\Sigma^2$, i.e. $|N|=0$.

Assume that $|N|>0$. Then there is a global orthonormal frame field $\big\{E_3=\frac{H}{|H|},E_4=\frac{N}{|N|}\big\}$, and we have $A_4=0$ and $|\sigma|^2=|A_3|^2=c(2-|T|^2)$.

On the
other hand, since the surface is flat, from \eqref{eq:R}, it follows
that
$$
0=2K=2c(1-|T|^2)+4|H|^2-|\sigma|^2=-c|T|^2+4|H|^2
$$
and then, that
\begin{equation}\label{eq:equality}
4|H|^2=c|T|^2.
\end{equation}

From Proposition \ref{p:delta} and Corollary \ref{coro:bih_pmc}, in
the same way as in the proof of Theorem~\ref{thm:gap}, we have
\begin{equation}\label{eq:eq}
\frac{1}{2}\Delta|A_H|^2=|\nabla A_H|^2+2(\trace A_H^3)-4c|H|^4.
\end{equation}
Next, since Corollary \ref{lemma:T=ct} implies that
$|A_H|^2=c(2-|T|^2)|H|^2$ is constant, and, from relation
\eqref{eq:equality}, it follows
$$
\begin{array}{ll}
\trace A_H^3&=\trace\phi_H^3+3|H|^2|\phi_H|^2+2|H|^6=3c(2-|T|^2)|H|^4-4|H|^6\\ \\&=6c|H|^4-16|H|^6,
\end{array}
$$
equation \eqref{eq:eq} leads to
$$
\begin{array}{ll}
0&=|\nabla A_H|^2+8c|H|^4-32|H|^6=|\nabla A_H|^2+8|H|^4(c-4|H|^2)\\ \\&=|\nabla A_H|^2+8c|H|^4|N|^2,
\end{array}
$$
and we get that $|\nabla A_H|^2=0$ and $N=0$. Therefore, the surface is a vertical cylinder. Now, since $\Sigma^2$ is flat,
we have $|H|=\frac{1}{2}\sqrt{c}$, i.e. $\Sigma^2=\pi^{-1}(\gamma)$, where $\gamma$ is a proper-biharmonic pmc curve in $\mathbb{S}^3(c)$
with curvature $\kappa=2|H|=\sqrt{c}$. It follows that $\gamma$ actually is a proper-biharmonic circle
in $\mathbb{S}^2(c)$.

Summarizing, we have the following rigidity result.

\begin{theorem}\label{thm:surfaces} Let $\Sigma^2$ be a proper-biharmonic pmc surface in $\mathbb{S}^n(c)\times\mathbb{R}$. Then either
\begin{enumerate}
\item $\Sigma^2$ is a minimal surface of a small hypersphere $\mathbb{S}^{n-1}(2c)\subset\mathbb{S}^n(c)$$;$ or

\item $\Sigma^2$ is $($an open part of$)$ a vertical cylinder $\pi^{-1}(\gamma)$, where $\gamma$ is
a circle in $\mathbb{S}^2(c)$ with curvature equal to $\sqrt{c}$, i.e. $\gamma$ is a biharmonic circle in $\mathbb{S}^2(c)$.
\end{enumerate}
\end{theorem}

From Theorem \ref{thm:surfaces} we can see that equation $\nabla
A_H=0$ holds for all proper-biharmonic surfaces. From this point of
view, the following result can be seen as a generalization of that
theorem for higher dimensional submanifolds. Before stating the
theorem, we have to mention that proper-biharmonic pmc submanifolds
in $\mathbb{S}^n(c)$, with $\nabla A_H=0$, were classified in
\cite{BO}.

\begin{theorem}\label{thm:general} If $\Sigma^m$, with $m\geq 3$, is a proper-biharmonic pmc submanifold in $\mathbb{S}^n(c)\times\mathbb{R}$ such that $\nabla
A_H=0$, then either
\begin{enumerate}
\item $\Sigma^m$ is a proper-biharmonic pmc submanifold in
$\mathbb{S}^n(c)$, with $\nabla A_H=0$$;$ or

\item $\Sigma^m$ is $($an open part of$)$ a vertical cylinder $\pi^{-1}(\Sigma^{m-1})$, where $\Sigma^{m-1}$ is a proper-biharmonic pmc submanifold in
$\mathbb{S}^n(c)$ such that the shape operator corresponding to
its mean curvature vector field in $\mathbb{S}^n(c)$ is
parallel.
\end{enumerate}
\end{theorem}

\begin{proof} On the one hand, since $\nabla
A_H=0$, we have that $[R(X,Y),A_H]=0$ for any vector fields $X$ and $Y$ tangent to $\Sigma^m$. On the other hand, since $\Sigma^m$ is a proper-biharmonic pmc submanifold, as we have seen in the proof of Theorem \ref{thm:gap2}, we also have $A_HT=0$, and then it follows that
$A_HR(X,T)T=0$.

Now, let us consider a local orthonormal frame field $\{E_i\}_{i=1}^{m}$ on $\Sigma^m$, and $\big\{E_{m+1}=\frac{H}{|H|},\ldots,E_{n+1}\big\}$ an orthonormal frame field in the normal bundle. We obtain, using \eqref{eq:R}, Lemma \ref{lemma:com} and again $A_HT=0$,
$$
\begin{array}{ll}
0&=\sum_{i=1}^{m}\langle A_HR(E_i,T)T,E_i\rangle\\ \\ &=c(\trace A_H)|T|^2(1-|T|^2)+\sum_{\alpha=m+2}^{n+1}(\trace(A_HA_{\alpha}))\langle A_{\alpha}T,T\rangle.
\end{array}
$$
From the last equation of \eqref{eq:bih_pmc}, we know that $\trace(A_HA_{\alpha})=0$ for $\alpha\in\{m+2,\ldots,n+1\}$, and therefore we have
$$
0=cm|H|^2|T|^2(1-|T|^2),
$$
i.e. either $|T|=0$ or $|T|=1$, which completes the proof.
\end{proof}

\end{document}